\documentclass[10pt]{amsart}

\def\bC {\mathbf{C}}

\def\bN {\mathbf{N}}

\def\bR {\mathbf{R}}

\def\cD {\mathcal{D}}

\def\cM {\mathcal{M}}

\def\a {{\alpha}}
\def\b {{\beta}}
\def\g {{\gamma}}

\def\eps {{\epsilon}}

\def\l {{\lambda}}

\def\si {{\sigma}}

\def\om {{\omega}}

\def\d {{\partial}}
\def\grad {{\nabla}}
\def\Dlt {{\Delta}}

\def\rstr {{\big |}}
\def\indc {{\bf 1}}

\newcommand{\Div}{\operatorname{div}}
\newcommand{\Rot}{\operatorname{curl}}

\newcommand{\Supp}{\operatorname{supp}}

\newcommand{\ba}{\begin{aligned}}
\newcommand{\ea}{\end{aligned}}

\newcommand{\be}{\begin{equation}}
\newcommand{\ee}{\end{equation}}

\newcommand{\lb}{\label}

\newtheorem{Thm}{Theorem}[section]

\newtheorem{Cor}[Thm]{Corollary}
\newtheorem{Lem}[Thm]{Lemma}

\begin{document}

\title[Optimal Regularity for Conservation Laws]{Optimal Regularizing Effect\\ for Scalar Conservation Laws}

\author[F. Golse]{Fran\c cois Golse}
\address[F.G.]{Ecole Polytechnique, Centre de Math\'ematiques L. Schwartz, 91128 Palaiseau Cedex France \& Universit\'e Paris-Diderot, Laboratoire 
J.-L. Lions, BP187, 4 place Jussieu, 75252 Paris Cedex 05 France}
\email{francois.golse@math.polytechnique.fr}

\author[B. Perthame]{Beno\^\i t Perthame}
\address[B.P.]{Universit\'e Pierre-et-Marie Curie, Laboratoire J.-L. Lions, BP187, 4 place Jussieu, 75252 Paris Cedex 05 France}
\email{benoit.perthame@upmc.fr}

\begin{abstract}
We investigate the regularity of bounded weak solutions of scalar conservation laws with uniformly convex flux in space dimension one, satisfying 
an entropy condition with entropy production term that is a signed Radon measure. We prove that all such solutions belong to the Besov space 
$B^{1/3,3}_{\infty,loc}$. Since C. DeLellis and M. Westdickenberg [Ann. Inst. H. Poincar\'e Anal. Non Lin\'eaire \textbf{20} (2003), 1075--1085] have 
proved the existence of such solutions that do not belong to $B^{s,p}_{q,loc}$ if either $s>1/\max(p,3)$ or $s=1/3$ and $1\le q<p<3$ or  $s=1/p$ 
with $p\ge 3$ and $q<\infty$, this regularizing effect is optimal. The proof is based on the kinetic formulation of scalar conservation laws and on an 
interaction estimate in physical space.
\end{abstract}



\maketitle

\section{Introduction}

Consider the Cauchy problem for the free transport equation with unknown $u\equiv u(t,x)\in\bR$
\be\lb{FreeTransp}
\left\{
\ba
{}&\d_tu+c\d_xu=0\,,\qquad x\in\bR\,,\,\,t>0\,,
\\
&u\rstr_{t=0}=u^{in}\,,
\ea
\right.
\ee
where $c\in\bR$ is a constant. It is well known that 
$$
u(t,x)=u^{in}(x-ct)\,,
$$
so that $u(t,\cdot)$ has exactly the same level of regularity as $u^{in}$.

If the speed of propagation $c$ depends on the unknown $u$, the situation is completely different. Consider the scalar conservation law with unknown
$u\equiv u(t,x)\in\bR$ and flux $a:\,\bR\to\bR$ of class $C^2$
\be\lb{ConsLaw}
\left\{
\ba
{}&\d_tu+\d_xa(u)=0\,,\qquad x\in\bR\,,\,\,t>0\,,
\\
&u\rstr_{t=0}=u^{in}\,.
\ea
\right.
\ee
If $u$ is of class $C^1$, the scalar conservation law (\ref{ConsLaw}) is equivalent to the free transport equation (\ref{FreeTransp}) with $c=a'(u)$. But 
even if $u^{in}\in C^1(\bR)$, it may happen that the solution $u$ of (\ref{ConsLaw}) loses the $C^1$ regularity in finite time. (This was already known
to Riemann: see \cite{Riemann}.)  More precisely, if $a$ is  convex, and if $u^{in}$ is decreasing on an open interval, there exists $T\in\bR_+^*$ such 
that (\ref{ConsLaw}) cannot have a $C^1$ solution defined on $(0,T')\times\bR$ for all $T'>T$. However, (\ref{ConsLaw}) has global weak solutions 
whose restrictions to $(T,\infty)\times\bR$ contains jump discontinuities. 

If $a''(v)\ge\a>0$ for each $v\in\bR$, for each $u^{in}\in L^1(\bR)$, there exists a unique weak solution $u\in L^\infty(\bR_+;L^1(\bR))$ of (\ref{ConsLaw})
satisfying in addition the differential inequality, referred to as the entropy condition
\be\lb{EntrCond}
\d_t\eta(u)+\d_xq(u)\le 0
\ee
for each convex $C^1$ function $\eta$, referred to as the entropy, where
\be\lb{Def-q}
q(v)=\int^v\eta'(w)a'(w)dw\,,
\ee
referred to as entropy flux associated to $\eta$.

This solution satisfies $u(t,\cdot)\in BV_{loc}\cap L^\infty(\bR)$ for each $t>0$ --- see formula (4.9) in chapter 4 of \cite{Lax73} for the $L^\infty$ bound, 
and  chapter I6 \S A in \cite{Smoller} for the $BV_{loc}$ bound. Thus, if $u^{in}\in C^1(\bR)$, since $u(t,\cdot)\in BV_{loc}(\bR)$ may contain jump 
discontinuities, the effect of the nonlinearity $a$ is a loss of regularity in the solution. Yet, if $u^{in}\in L^1(\bR)$, the fact that $u(t,\cdot)\in BV_{loc}(\bR)$ 
for each $t>0$ can be viewed as a (limited) regularizing effect. 

The purpose of the present paper is to study the optimal regularizing effect for weak solutions of (\ref{ConsLaw}) satisfying the weaker entropy condition
\be\lb{WEntrCond}
\d_t\eta(u)+\d_xq(u)=-\mu
\ee
for each entropy-entropy flux pair $(\eta,q)$ as above, i.e. satisfying (\ref{Def-q}), where $\mu$ is a \textit{signed} Radon measure on $\bR_+^*\times\bR$ 
(instead of a positive measure). Such solutions may contain jump discontinuities that would dissipate instead of create entropy and therefore would 
be considered as unphysical by analogy with gas dynamics. Yet, such solutions are relevant for other physical applications, such as micromagnetism: 
see for instance \cite{Riviere} and the references therein.

Another motivation for considering the entropy condition (\ref{WEntrCond}) with entropy production $\mu$ of indefinite sign can be found in the work of
Hwang and Tzavaras \cite{HwangTzav}. In this work, the authors show that two different strategies for approximating solutions of scalar conservation
laws, the relaxation approximation \textit{\`a la} Jin-Xin \cite{JinXin} and the diffusion-dispersion approximation \textit{\`a la} Schonbek \cite{Schonbek}, 
may lead to kinetic formulations involving entropy production measures that may in general fail to be positive. 

The regularizing effect for this type of solutions of (\ref{ConsLaw}) has been studied so far by using a kinetic formulation of the scalar conservation law:
see \cite{LPT-a} for the original contribution, \cite{JabPerth} for an improved regularity result, and \cite{PerthBook} for a detailed presentation of kinetic
formulations and their properties. The tool for establishing the regularizing effect for kinetic formulations is a class of results known as velocity averaging, 
introduced independently in \cite{Agosh} and \cite{GPS}, with subsequent generalizations and improvements described for instance in chapter 1 of 
\cite{BouGolPul} --- see also the list of references given in section \ref{S-VA}. 

The best result obtained by this method is that $u\in W^{s,r}_{loc}(\bR_+^*\times\bR)$ for all $s<1/3$ and $1\le r<3/2$ (see \cite{JabPerth}) --- in the 
earlier result \cite{LPT-a}, the integrability exponent was restricted to $r<5/3$. On the other hand, it is possible to construct  such solutions that do not
belong to any Besov space ``better than'' $B^{1/3,3}_{\infty,loc}(\bR_+^*\times\bR)$ --- see section \ref{S-ConsLaw} for a more precise statement of 
this optimality result and \cite{DeLelWest} for the proof.

In the present paper, we prove that solutions of  (\ref{ConsLaw}) with entropy production that is a signed Radon measure for each convex entropy does 
indeed satisfy a (local) $B^{1/3,3}_\infty$ estimate provided that the flux function is uniformly convex. With the result in \cite{DeLelWest}, this shows that 
the optimal regularity space for such solutions is indeed $B^{1/3,3}_{\infty,loc}(\bR_+^*\times\bR)$.

\begin{Thm}
Assume that $a\in C^2(\bR)$ satisfies $a''\ge\a_0>0$ on $\bR$, and let $u^{in}\in L^\infty(\bR)$. Any bounded weak solution of (\ref{ConsLaw})  
satisfying (\ref{WEntrCond}) for each convex entropy $\eta$, with an entropy production $\mu$ that is a signed Radon measure on $\bR_+\times\bR$, 
belongs to $B^{1/3,3}_{\infty,loc}(\bR_+^*\times\bR)$.
\end{Thm}

Our proof of this result relies on a method completely different from velocity averaging. At variance with velocity averaging, this method does not use any 
argument from harmonic analysis (Fourier transform, Littlewood-Paley decompositions...), but is based on an ``interaction identity'' presented in section 
\ref{S-Interact}. After recalling an earlier, weaker regularizing effect obtained in \cite{GolseHyp08, GolseICMP09} for solutions of (\ref{ConsLaw}) subject 
to a weaker entropy condition  --- Theorem \ref{T-RegFG} in section \ref{S-ConsLaw} --- the optimal regularizing effect is stated as Theorem 
\ref{T-RegEffCvx} in that same section. As a warm-up, we use our method based on the interaction identity to establish a new velocity averaging result --- 
Theorem \ref{T-VelAv} in section \ref{S-VA}. Unlike in all velocity averaging theorems known to this date, the regularity of  velocity averages of the solution 
of the transport equation obtained in this result is independent of the regularity of the source term in the velocity variable. The tradeoff is that the kinetic 
solution must satisfy a seemingly unnatural monotonicity condition --- see condition (\ref{Hyp-f}) below --- which however  turns out to be satisfied precisely 
in the context of the kinetic formulation of scalar conservation laws.

This new method is reminiscent of some tools from compensated compactness \cite{Murat, TartarHW79}--- see section \ref{S-Interact} below and the 
comments in section 2 of \cite{GolseHyp08}. A shortcoming of this method is that, unlike velocity averaging, it seems confined to the case of one space 
dimension, at least at the time of this writing. In the case of space dimension higher than one, some regularizing effect in the time variable of special 
classes of nonlinear fluxes (that are in particular power-like at infinity) has been established in \cite{OttoTransp} --- an earlier result of same type in the
case of homogeneous nonlinearities can be found in \cite{BeniCran}.

\section{The Interaction Identity}\lb{S-Interact}

Consider the system of partial differential relations
\be\lb{2x2Syst}
\left\{
\ba
\d_tA+\d_xB=C\,,
\\
\d_tD+\d_xE=F\,,
\ea
\right.
\ee
where $A,B,C,D,E,F$ are real-valued functions of $t\ge 0$ and $x\in\bR$. We henceforth assume that the functions $A,B,C,D,E,F$ have compact support 
in $\bR_+^*\times\bR$, and are extended by $0$ to $\bR\times\bR$.

The quantity
$$
I_1(t):=\iint_{x<y}A(t,x)D(t,y)dxdv
$$
has been introduced by S.R.S. Varadhan --- see Lemma 22.1 in \cite{TartarKin} --- in the context of discrete velocity models in kinetic theory, and used
by several authors since then (J.-M. Bony \cite{Bony}, C. Cercignani \cite{CerciWeakSol}, S.-Y. Ha \cite{Ha}). It is also reminiscent of Glimm's interaction 
functional used in the theory of hyperbolic systems of conservation laws in space dimension $1$: see \cite{Glimm65}.

Differentiating under the integral sign and using both equations in (\ref{2x2Syst}), one obtains the following interaction identity
$$
\ba
\frac{dI_1}{dt}(t)&=\iint_{x<y}(C(t,x)D(t,y)+A(t,x)F(t,y))dxdy
\\
&-\iint_{x<y}(\d_xB(t,x)D(t,y)+A(t,x)\d_yE(t,y))dxdy
\\
&=\iint_{x<y}(C(t,x)D(t,y)+A(t,x)F(t,y))dxdy
\\
&+\int_{\bR}(AE-DB)(t,z)dz\,.
\ea
$$
Integrating further in the time variable and observing that $I_1$ is compactly supported in $t\in\bR_+^*$, we arrive at the identity
\be\lb{InterIdent}
\ba
\iint_{\bR\times\bR}(AE-DB)(t,z)dzdt=&-\iint_{\bR\times\bR}C(t,x)\left(\int_x^\infty D(t,y)dy\right)dxdt
\\
&-\iint_{\bR\times\bR}F(t,y)\left(\int_{-\infty}^y A(t,x)dx\right)dydt\,.
\ea
\ee

Exchanging the roles of the time and space variables in the computation above, one can consider the quantity
$$
I_2(t):=\iint_{s<t}B(s,z)E(t,z)dsdt
$$
instead of $I_1$. Proceeding as above, one sees that
$$
\ba
\frac{dI_2}{dz}(z)&=\iint_{s<t}(C(s,z)E(t,z)+B(s,z)F(t,z))dsdt
\\
&-\iint_{s<t}(\d_sA(s,z)E(t,z)+B(s,z)\d_tD(t,z))dsdt
\\
&=\iint_{s<t}(C(s,z)E(t,z)+B(s,z)F(t,z))dsdt
\\
&-\int_{\bR}(AE-DB)(t,z)dt\,.
\ea
$$
Integrating further in $z$ and observing that $I_2$ is compactly supported in $z\in\bR$, we obtain
\be\lb{InterIdent2}
\ba
\iint_{\bR\times\bR}(AE-DB)(t,z)dt=&\iint_{\bR\times\bR}C(s,z)\left(\int_s^\infty E(t,z)dt\right)dzds
\\
&+\iint_{\bR\times\bR}F(t,z)\left(\int_{-\infty}^tB(s,z)ds\right)dzdt\,.
\ea
\ee

In his proof of the interaction identity on p. 182 of his book \cite{TartarKin}, L. Tartar observes that the structure of this identity is reminiscent of compensated 
compactness \cite{Murat, TartarHW79}. Indeed, introducing the vector fields
$$
U(t,x,y):=(A(t,x),B(t,x),0)\,,\qquad V(t,x,y)=(E(t,x),-B(t,x),0)
$$
we see that the system (\ref{2x2Syst}) takes the form
$$
\left\{
\begin{array}{l}
\Div \,U\,=C\,,
\\
\Rot V=(0,0,-F)\,.
\end{array}
\right.
$$
While the left hand side of (\ref{InterIdent}) involves the inner product $U\cdot V$, the right hand side involves integrands in the form of binary products
where one of the terms is integrated --- and therefore gains one order of regularity --- in the space variable. The same is true of (\ref{InterIdent2}), by
which one can hope to gain one order of regularity in the time variable.

In view of this observation, the idea of using the interaction identities (\ref{InterIdent})-(\ref{InterIdent2}) for the purpose of establishing regularization
result appears fairly natural, and will be used systematically in the sequel.

\section{Velocity Averaging in Physical Space}\lb{S-VA}

Let $f\equiv f(t,x,v)$ satisfy
\be\lb{KinEq}
(\d_t+a'(v)\d_x)f=\d_v^\g m
\ee
where $a$ is a smooth function while $m$ is a bounded, signed Radon measure on $\bR_+\times\bR\times\bR$, and $\g\in\bN$. 

Transport equations of this type naturally appear in the kinetic formulation of hyperbolic systems of conservation laws: see for instance \cite{PerthBook}.
A fundamental question in the context of kinetic models is to investigate the local regularity of moments in the velocity variable $v$ of the function $f$.
Systematic investigations on this class type of questions began with our work with R. Sentis \cite{GPS} (see also the independent study by V. Agoshkov
\cite{Agosh}) and in a series of subsequent contributions by several authors where more and more general classes of functions $f$ and right hand sides 
are considered: see in particular \cite{GLPS,dPLM,BouDesv,DVorePe,PerthSoug,GSR,TadTao,JabPerth,JabVega,BerthJun,ArseSR}. 

All these works use at some point tools from harmonic analysis: Fourier transform, Hardy-Littlewood decomposition, Radon transform. Moreover, in all
these results, the regularity of moments in $v$ of the function $f$ depends on $\g$. 

In this section, we give an example of velocity averaging result where the regularity of the moments in $v$ of $f$ is independent of $\g$, at the expense 
of an extra assumption on the $v$ dependence in $f$. Also, the proof of this result is based on the interaction identities (\ref{InterIdent})-(\ref{InterIdent2})
and uses only elementary techniques in physical space.

\begin{Thm}\lb{T-VelAv}
Let $a\in C^1(\bR)$; assume that there exists $\b\ge 1$ such that, for each $M>0$, there exists $\a_M>0$ for which
\be\lb{Hyp-a}
a'(v)-a'(w)\ge\a_M(v-w)^\b\,,\qquad -M\le w<v\le M\,.
\ee
Let $\g\in\bN$ and let $m$ be a signed Radon measure on $\bR_+\times\bR\times\bR$. 

Assume that $f\in L^\infty(\bR_+\times\bR\times\bR)$ satisfies (\ref{KinEq}) and that, for each $y\in\bR$ and each $s\ge 0$
\be\lb{Hyp-f}
\ba
(f(t+s,x+y,v)-f(t,x,v))(f(t+s,x+y,w)-f(t,x,w))\ge 0&
\\
\hbox{ for a.e. }(t,x,v,w)\in\bR_+\times\bR\times\bR\times\bR&\,.
\ea
\ee
Then, for each $\psi\in C^\infty_c(\bR)$, one has
$$
\int_{\bR}f\psi(v)dv\in B^{s,2}_{\infty,loc}(\bR_+^*\times\bR)\quad\hbox{ with }s=\frac1{4+2\b}\,.
$$
\end{Thm}

We recall that, in the context of velocity averaging, the gain of regularity obtained on averages of the form
$$
\int_{|\xi|\le M}f(t,x,\xi)d\xi
$$
of solutions of the transport equation
$$
(\d_t+V(\xi)\cdot\grad_x)f(t,x,\xi)=g(t,x,\xi)
$$
involves  
$$
\sup_{\om^2+|k|^2=1}|\{\xi\in\bR^N\hbox{ s.t. }|\xi|\le M\hbox{ and }|\om+V(\xi)\cdot k|\le\eps\}|\,,
$$
where, for each $A\subset\bR^N$, the Lebesgue measure of $A$ is designated by $|A|$  --- see condition (2.1) in \cite{GLPS}. Equivalently, whenever 
$k\not=0$,
$$
\ba
|\{\xi\in\bR^N\hbox{ s.t. }|\xi|\le M\hbox{ and }|\om+V(\xi)\cdot k|\le\eps\}|&
\\
=
\left|\left\{\xi\in\bR^N\hbox{ s.t. }|\xi|\le M\hbox{ and }V(\xi)\cdot \frac{k}{|k|}\in\left[-\frac{\om+\eps}{|k|},\frac{\eps-\om}{|k|}\right]\right\}\right|&\,.
\ea
$$
On the other hand, (\ref{Hyp-a}) implies that $a$ is strictly convex and is equivalent to the condition 
$$
|\{v\in\bR\hbox{ s.t. }|v|\le M\hbox{ and }a'(v)\in[A,B]\}|=a'^{-1}(B)-a'^{-1}(A)\le\left(\frac{B-A}{\a_M}\right)^{1/\b}
$$
for each $A,B$ such that $a'(-M)\le A<B\le a'(M)$. Therefore, (\ref{Hyp-a}) is a condition of the same type as the classical condition used in velocity
averaging.

A typical sufficient condition under which $a$ satisfies (\ref{Hyp-a}) is as follows. Assume that $a\in C^{2n}(\bR)$ is convex and and that, for some 
$n\in\bN^*$ and $z\in(-M,M)$, one has
$$
a''(z)=\ldots=a^{(2n-1)}(z)=0\,,\quad\hbox{ and }a^{(2n)}(z)>0\,.
$$
(An example of this situation is the case of
$$
a(v):=\frac1{2n}v^{2n}\,,
$$
with $z=0$.) By continuity of $a^{(2n)}$, there exists $\rho>0$ such that $[z-\rho,z+\rho]\subset(-M,M)$ and $\l>0$ such that
$$
a^{(2n)}(t)\ge\l>0\quad\hbox{ for each }t\in[z-\rho,z+\rho]\,.
$$
By Taylor's formula, whenever $z-\rho\le w\le v\le z+\rho$, one has 
$$
\ba
a'(v)-a'(w)&=\int_w^v\frac1{(2n-2)!}((v-t)^{2n-2}\indc_{t>z}+(t-w)^{2n-2}\indc_{t<z})a^{(2n)}(t)dt
\\
&\ge\frac{\l}{(2n-1)!}((v-z)^{2n-1}+(z-w)^{2n-1})
\\
&\ge\frac{\l}{(2n-1)!}\max((v-z)^{2n-1},(z-w)^{2n-1})
\\
&\ge\frac{\l}{2^{2n-1}(2n-1)!}(v-w)^{2n-1}\,.
\ea
$$
On the other hand, whenever $-M\le w<z-\rho<z+\rho<v\le M$, one has
$$
\ba
a'(v)-a'(w)&=a'(v)-a'(z+\rho)+a'(z+\rho)-a'(z-\rho)+a'(z-\rho)-a'(w)
\\
&\ge a'(z+\rho)-a'(z-\rho)
\\
&\ge \frac{\l}{2^{2n-1}(2n-1)!}(2\rho)^{2n-1} 
\\
&\ge \frac{\l}{2^{2n-1}(2n-1)!}\left(\frac{\rho}{M}\right)^{2n-1}(v-w)^{2n-1}\,,
\ea
$$
where the first inequality follows from the convexity of $a$, while the second is a consequence of the previous estimate in the case $v=z+\rho$
and $w=z-\rho$. Therefore $a$ satisfies (\ref{Hyp-a}) with $\b=2n-1$ and $\a_M=\frac{\l}{2^{2n-1}(2n-1)!}\left(\frac{\rho}{M}\right)^{2n-1}$. This
situation is essentially the same as the one described in assumption (H) before Theorem 2.2 in \cite{GolseHyp08}. 

The assumption (\ref{Hyp-f}) may seem somewhat unnatural; however, as we shall see below, it is relevant in the context of conservation laws. A typical
example of function $f$ satisfying (\ref{Hyp-f}) is as follows. 

Let $\rho\in L^\infty(\bR_+\times\bR)$, and let $W:\,\bR\to\bR$ be a nondecreasing or nonincreasing function; then the function $f$ defined by the formula
$$
f(t,x,v)=W(\rho(t,x)-v)
$$
satisfies (\ref{Hyp-f}). Indeed, for each $t,s\ge 0$ and each $x,y\in\bR$ such that $\rho(t,x)$ and $\rho(s,y)$ are defined
$$
f(t,x,v)-f(s,y,v)=\int_{\rho(s,y)}^{\rho(t,x)}W'(u-v)du
$$
so that, assuming without loss of generality that $W$ is nondecreasing, we see that
$$
\rho(t,x)\ge\rho(s,y)\Rightarrow f(t,x,v)\ge f(s,y,v)\hbox{ for each }v\in\bR\,.
$$
Equivalently
$$
\ba
f(t,x,v)>f(s,y,v)\hbox{ for some }v\in\bR&\Rightarrow\rho(t,x)>\rho(s,y)
\\
&\Rightarrow f(t,x,w)\ge f(s,y,w)\hbox{ for all }w\in\bR\,.
\ea
$$
Therefore
$$
(f(t,x,v)-f(s,y,v))(f(t,x,w)-f(s,y,w))\ge 0\hbox{ for each }v,w\in\bR
$$
for all $t,s\in\bR_+$ and all $x,y\in\bR$ such that $\rho(t,x)$ and $\rho(s,y)$ are defined, which means that assumption (\ref{Hyp-f}) is satisfied in
this example.

\smallskip
We do not claim that the $B^{s,2}_{\infty,loc}$ regularity obtained in Theorem \ref{T-VelAv} is optimal. Since Theorem \ref{T-VelAv} is not the main 
result in the present paper, but rather an illustration of how to use the interaction identities (\ref{InterIdent})-(\ref{InterIdent2}) for the purpose of 
obtaining a velocity averaging theorem where the regularity index $s=\frac{1}{4+2\b}$ is independent of the number $\g$ of derivatives in $v$ in 
the source term, we have left this question aside.

\begin{proof}
For $h\in\bR$, define the operators $\cD_t^h$ and $\cD_x^h$ by the formulas
\be\lb{Def-Dh}
\left\{
\begin{array}{l}
\cD_t^h\phi(t,x):=\phi(t+h,x)-\phi(t,x)\,,
\\ \\
\cD_x^h\phi(t,x):=\phi(t,x+h)-\phi(t,x)\,.
\end{array}
\right.
\ee

Pick $\chi\in C^\infty_c(\bR_+^*\times\bR)$. One seeks to estimate
$$
\iint\chi(t,x)^2\left|\cD_x^h\int f(t,x,v)\psi(v)dv\right|^2dxdt\,.
$$
This quantity is decomposed as
$$
\ba
\iint\chi(t,x)^2\left|\cD_x^h\int f(t,x,v)\psi(v)dv\right|^2dxdt
\\
=
\iint\iint\chi(t,x)^2\cD_x^hf(t,x,v)\cD_x^hf(t,x,w)\psi(v)\psi(w)dvdwdxdt
\\
=
\iint\iint(1-\phi_\eps(w-v))\chi(t,x)^2\cD_x^hf(t,x,v)\cD_x^hf(t,x,w)\psi(v)\psi(w)dvdwdxdt
\\
+
\iint\iint\phi_\eps(w-v)\chi(t,x)^2\cD_x^hf(t,x,v)\cD_x^hf(t,x,w)\psi(v)\psi(w)dvdwdxdt
\\
=J_1+J_2\,,
\ea
$$
where $\phi_\eps(z)=\Phi(z/\eps)$ and $\Phi\in C^\infty(\bR)$ satisfies
$$
0\le\Phi(z)\le 1\,,\quad\Phi(z)=1\hbox{ if }|z|\ge 2\hbox{ and }\Phi(z)=0\hbox{ if }|z|\le 1\,.
$$

By assumption (\ref{Hyp-f}), the integral $J_1$ is estimated as follows:
\be\lb{J1<}
\ba
|J_1|\le\iint\iint_{|v-w|\le 2\eps}\chi(t,x)^2\cD_x^hf(t,x,v)\cD_x^hf(t,x,w)\psi(v)\psi(w)dvdwdxdt&
\\
\le 4\|\chi^2\|_{L^1}\|f\|^2_{L^\infty}\iint_{|v-w|\le 2\eps}\psi(v)\psi(w)dvdw&
\\
\le C_0\eps&\,,
\ea
\ee
where
$$
C_0:=16\|\chi^2\|_{L^1}\|f\|^2_{L^\infty}\|\psi\|_{L^1}\|\psi\|_{L^\infty}\,.
$$

As for the integral $J_2$, using again assumption (\ref{Hyp-f}) shows that
$$
|J_2|\le c_\eps J_3
$$
where, by (\ref{Hyp-a}),
$$
c_\eps=\sup_{-V\le w<v\le V}\frac{\phi_\eps(w-v)}{(v-w)(a'(v)-a'(w))}\le\frac1{\a_V\eps^{1+\b}}\,,
$$
assuming without loss of generality that $\Supp(\psi)\subset[-V,V]$, and
$$
J_3\!=\!\iint\!\!\!\iint(v\!-\!w)(a'(v)\!-\!a'(w))\chi(t,x)^2\cD_x^hf(t,\!x,\!v)\cD_x^hf(t,\!x,\!w)\psi(v)\psi(w)dvdwdxdt.
$$ 

The integral $J_3$ is now estimated by the interaction identity (\ref{InterIdent}). 

Set
$$
A(t,x,v):=\chi(t,x)\cD_x^hf(t,x,v)\,,\qquad D(t,x,w):=\chi(t,x)\cD_x^hf(t,x,w)\,.
$$
Since $f\equiv f(t,x,v)$ satisfies (\ref{KinEq}), the functions $A$ and $D$ defined above satisfy (\ref{2x2Syst}) with
$$
B(t,x,v):=a'(v)\chi(x)\cD_x^hf(t,x,v)\,,\qquad E(t,x,w):=a'(w)\chi(x)\cD_x^hf(t,x,w)\,,
$$
and
$$
\left\{
\ba
C(t,x,v)&:=\chi(t,x)\d_v^{\g}\cD_x^hm(t,x,v)\,+X(t,x,v)\cD_x^hf(t,x,v)\,,
\\
F(t,x,w)&:=\chi(t,x)\d_w^{\g}\cD_x^hm(t,x,w)+X(t,x,w)\cD_x^hf(t,x,w)\,,
\ea
\right.
$$
with the notation
$$
X(t,x,v)=(\d_t\chi+a'(v)\d_x\chi)(t,x)\,.
$$
(We have abused the notation $m(t,x,v)$ as if the signed Radon measure $m$ was a function.)

At this point, we apply the interaction identity (\ref{InterIdent}). After multiplying both sides of this identity by $(v-w)\psi(v)\psi(w)$ and integrating in $v,w$,
one obtains
\be\lb{J3=}
\ba
J_3=&-\iint(v-w)\psi(v)\psi(w)\iint C(t,x,v)\left(\int_x^\infty D(t,y,w)dy\right)dxdtdvdw
\\
&-\iint(v-w)\psi(v)\psi(w)\iint F(t,y,w)\left(\int_{-\infty}^y A(t,x,v)dx\right)dydtdvdw\,.
\ea
\ee
Each integral in the right hand side of $J_3$ involves two different kinds of terms. One is
\be\lb{J31=}
J_{31}=\!\iint(v\!-\!w)\psi(v)\psi(w)\iint X(t,x,v)\cD_x^hf(t,x,v)\left(\int_x^\infty\!\! D(t,y,w)dy\right)dxdtdvdw\,,
\ee
the other being
\be\lb{J32=}
J_{32}=\!\iint(v\!-\!w)\psi(v)\psi(w)\!\iint\chi(t,x)\d_v^{\g}\cD_x^hm(t,x,v)\left(\int_x^\infty\!\! D(t,y,w)dy\right)dxdtdvdw\,.
\ee

In both expressions, the inner integral is put in the form
$$
\ba
\int_x^\infty D(t,y,w)dy=\int_x^\infty\cD_y^h(\chi(t,y)f(t,y,w))dy-\int_x^\infty f(t,y+h,w)\cD_y^h\chi(t,y)dy&
\\
=-\int_{x}^{x+h}\chi(t,y)f(t,y,w)dy-\int_x^\infty f(t,y+h,w)\left(\int_0^h\d_x\chi(t,y+z)dz\right)dy&\,,
\ea
$$
so that
\be\lb{EstimD}
\left|\int_x^\infty D(t,y,w)dy\right|\le\|f\|_{L^\infty}(\|\chi\|_{L^\infty}+\|\d_x\chi\|_{L^1})|h|\,.
\ee

Thus
$$
|J_{31}|\le C_1|h|
$$
with
$$
C_1=2\|f\|^2_{L^\infty}(\|\chi\|_{L^\infty}+\|\d_x\chi\|_{L^1})(\|vX\psi\|_{L^1}\|\psi\|_{L^1}+\|X\psi\|_{L^1}\|v\psi\|_{L^1})\,.
$$

In $J_{32}$, we first integrate by parts and bring the $v$ derivatives to bear on the weight $(v-w)\psi(v)$:
$$
\ba
J_{32}=(-1)^\g\iint\d_v^{\g}((v-w)\psi(v))\psi(w)\iint\chi(t,x)\cD_x^hm(t,x,v)&
\\
\times\left(\int_x^\infty D(t,y,w)dy\right)dxdtdvdw&\,.
\ea
$$
Assuming without loss of generality that $\Supp(\chi)\subset[0,T]\times[-R,R]$ and recalling that $\Supp(\psi)\subset[-V,V]$ while $|h|\le 1$, one obtains
$$
|J_{32}|\le C_2|h|
$$
with
$$
\ba
C_2=2\|f\|_{L^\infty}(\|\d^\g(v\psi)\|_{L^\infty}\|\psi\|_{L^1}+\|\d^\g\psi\|_{L^\infty}\|v\psi\|_{L^1})&
\\
\times(\|\chi\|_{L^\infty}+\|\d_x\chi\|_{L^1})\|\chi\|_{L^\infty}\int_0^T\int_{-R-1}^{R+1}\int_{-V}^V|m|&\,.
\ea
$$

In conclusion
$$
|J_3|\le 2(C_1+C_2)|h|\,.
$$

Therefore
$$
\iint\chi(t,x)^2\left|\cD_x^h\int f(t,x,v)\psi(v)dv\right|^2dxdt\le C_0\eps+\tfrac2{\a_V}(C_1+C_2)\frac{|h|}{\eps^{1+\b}}\,,
$$
and choosing $\eps=|h|^{1/(2+\b)}$, we find that
\be\lb{RegVAx}
\iint\chi(t,x)^2\left|\cD_x^h\int f(t,x,v)\psi(v)dv\right|^2dxdt\le\left(C_0+\tfrac2{\a_V}(C_1+C_2)\right)|h|^{1/(2+\b)}\,.
\ee

As for the time regularity, it is obtained similarly, exchanging the roles of the variables $t$ and $x$. We briefly sketch the argument below. One seeks to 
estimate, for each $h>0$, the quantity
$$
\iint\chi(t,x)^2\left|\cD_t^h\int f(t,x,v)\psi(v)dv\right|^2dxdt
$$
that is decomposed as
$$
\ba
\iint\chi(t,x)^2\left|\cD_t^h\int f(t,x,v)\psi(v)dv\right|^2dxdt
\\
=
\iint\iint\chi(t,x)^2\cD_t^hf(t,x,v)\cD_t^hf(t,x,w)\psi(v)\psi(w)dvdwdxdt
\\
=
\iint\iint(1-\phi_\eps(w-v))\chi(t,x)^2\cD_t^hf(t,x,v)\cD_t^hf(t,x,w)\psi(v)\psi(w)dvdwdxdt
\\
+
\iint\iint\phi_\eps(w-v)\chi(t,x)^2\cD_t^hf(t,x,v)\cD_t^hf(t,x,w)\psi(v)\psi(w)dvdwdxdt
\\
=K_1+K_2\,.
\ea
$$

The term $K_1$ is obviously estimated exactly as $J_1$:
\be\lb{K1<}
|K_1|\le 4\|\chi^2\|_{L^1}\|f\|^2_{L^\infty}\iint_{|v-w|\le 2\eps}\psi(v)\psi(w)dvdwdxdt\le C_0\eps\,.
\ee

As in the case of $J_2$, one has
$$
|K_2|\le c_\eps K_3
$$
where
$$
K_3\!=\!\iint\!\!\!\iint(v\!-\!w)(a'(v)\!-\!a'(w))\chi(t,x)^2\cD_t^hf(t,\!x,\!v)\cD_t^hf(t,\!x,\!w)\psi(v)\psi(w)dvdwdxdt.
$$ 
This last integral is estimated by using the interaction identity (\ref{InterIdent2}), with a slightly different definition of $A,B,C,D,E,F$:
$$
A(t,x,v):=\chi(t,x)\cD_t^hf(t,x,v)\,,\qquad D(t,x,w):=\chi(t,x)\cD_t^hf(t,x,w)\,,
$$
while
$$
B(t,x,v):=a'(v)\chi(x)\cD_t^hf(t,x,v)\,,\qquad E(t,x,w):=a'(w)\chi(x)\cD_t^hf(t,x,w)
$$
and
$$
\left\{
\ba
C(t,x,v)&=\chi(t,x)\d_v^{\g}\cD_t^hm(t,x,v)\,+X(t,x,v)\cD_t^hf(t,x,v)\,,
\\
F(t,x,w)&=\chi(t,x)\d_w^{\g}\cD_t^hm(t,x,w)+X(t,x,w)\cD_t^hf(t,x,w)\,.
\ea
\right.
$$

Apply the interaction identity (\ref{InterIdent2}) after multiplying both sides of this identity by $(v-w)\psi(v)\psi(w)$ and integrating in $v,w$:
\be\lb{K3=}
\ba
K_3=&-\iint(v-w)\psi(v)\psi(w)\iint C(s,x,v)\left(\int_s^\infty E(t,x,w)dt\right)dxdsdvdw
\\
&-\iint(v-w)\psi(v)\psi(w)\iint F(t,x,w)\left(\int_{-\infty}^t B(s,x,v)ds\right)dxdtdvdw\,.
\ea
\ee
As in the case of $J_3$, the right hand side of $K_3$ involves two different kinds of terms:
\be\lb{K31=}
K_{31}=\!\iint(v\!-\!w)\psi(v)\psi(w)\!\iint X(t,x,v)\cD_s^hf(s,x,v)\left(\int_s^\infty\!\!  E(t,x,w)ds\right)dxdsdvdw
\ee
and
\be\lb{K32=}
K_{32}=\!\iint(v\!-\!w)\psi(v)\psi(w)\iint\chi(t,x)\d_v^{\g}\cD_s^hm(s,x,v)\left(\int_s^\infty\!\!  E(t,x,w)dt\right)dxdsdvdw\,.
\ee
The inner integral is put in the form
$$
\ba
\int_s^\infty E(t,x,w)dt=&\int_s^\infty\cD_s^h(\chi(t,x)a'(w)f(t,x,w))dy
\\
&-\int_s^\infty a'(w)f(t+h,x,w)\cD_t^h\chi(t,x)dt
\\
=&-\int_{s}^{s+h}\chi(t,x)a'(w)f(t,x,w)dt
\\
&-\int_s^\infty a'(w)f(t+h,x,w)\left(\int_0^h\d_t\chi(t+\tau,x)d\tau\right)dt\,,
\ea
$$
so that
\be\lb{EstimE}
\left|\int_s^\infty E(t,x,w)dt\right|\le|a'(w)|\|f\|_{L^\infty}(\|\chi\|_{L^\infty}+\|\d_t\chi\|_{L^1})h\,.
\ee

Thus
$$
|K_{13}|\le C_3h\,,
$$
with
$$
C_3=2\|f\|^2_{L^\infty}(\|\chi\|_{L^\infty}+\|\d_t\chi\|_{L^1})(\|vX\psi\|_{L^1}\|a'\psi\|_{L^1}+\|X\psi\|_{L^1}\|va'\psi\|_{L^1})\,.
$$

Next
$$
\ba
K_{32}=(-1)^\g\iint\d_v^{\g}((v-w)\psi(v))\psi(w)\iint\chi(t,x)\cD_s^hm(s,x,v)&
\\
\times\left(\int_s^\infty E(t,x,w)dt\right)dxdtdvdw&\,.
\ea
$$
Assuming that $\Supp(\chi)\subset[0,T]\times[-R,R]$ while $\Supp(\psi)\subset[-V,V]$  without loss of generality, and that $0<h\le 1$, one obtains
$$
|K_{32}|\le C_4h
$$
with
$$
\ba
C_4=2\|f\|_{L^\infty}(\|\d^\g(v\psi)\|_{L^\infty}\|a'\psi\|_{L^1}+\|\d^\g\psi\|_{L^\infty}\|va'\psi\|_{L^1})&
\\
\times(\|\chi\|_{L^\infty}+\|\d_t\chi\|_{L^1})\|\chi\|_{L^\infty}\int_0^{T+1}\int_{-R}^{R}\int_{-V}^V|m|&\,.
\ea
$$

In conclusion
$$
|K_3|\le 2(C_3+C_4)|h|\,,
$$
so that
$$
\iint\chi(t,x)^2\left|\cD_t^h\int f(t,x,v)\psi(v)dv\right|^2dxdt\le C_0\eps+\tfrac2{\a_V}(C_3+C_4)\frac{h}{\eps^{1+\b}}\,,
$$
and choosing $\eps=h^{1/(2+\b)}$, we find that
\be\lb{RegVAt}
\iint\chi(t,x)^2\left|\cD_t^h\int f(t,x,v)\psi(v)dv\right|^2dxdt\le\left(C_0+\tfrac2{\a}(C_3+C_4)\right)h^{1/(2+\b)}\,.
\ee

Putting together (\ref{RegVAt}) and (\ref{RegVAx}), we conclude that
$$
\int f\psi(v)dv\in B^{s,2}_{\infty,loc}(\bR_+^*\times\bR)\quad\hbox{ with }s=\frac1{4+2\b}\,.
$$
\end{proof}

\section{Optimal Regularizing Effect for Scalar Conservation Laws in Space Dimension One}\lb{S-ConsLaw}

Consider the scalar conservation law
\be\lb{ScalCons}
\left\{
\begin{array}{l}
\d_tu+\d_xa(u)=0\,,\qquad x\in\bR\,,\,\,t>0\,,
\\	\\
u\rstr_{t=0}=u^{in}\,.
\end{array}
\right.
\ee
Assume that $u$ is a weak solution whose entropy production rate is a signed Radon measure. Specifically, we mean that, for each convex entropy 
$\eta\in C^1(\bR)$, one has
$$
\d_t\eta(u)+\d_xq(u)=-\int_{\bR}\eta''(v)dm(\cdot,\cdot,v)
$$
where $m$ is a signed Radon measure on $\bR_+\times\bR\times\bR$ (with compact support in the variable $v$) and the entropy flux $q$ is defined 
by (\ref{Def-q}).

Equivalently (see for instance \S 6.7 in \cite{Dafermos}), $u$ satisfies the following kinetic formulation of (\ref{ScalCons})
\be\lb{KinForm}
\left\{
\begin{array}{l}
\d_tf+a'(v)\d_xf=\d_vm\,,\qquad x,v\in\bR\,,\,\,t>0\,,
\\	\\
f\rstr_{t=0}=f^{in}\,,
\end{array}
\right.
\ee
where
\be\lb{Def-f}
f(t,x,v):=\left\{\begin{array}{ll}
+\indc_{[0,u(t,x)]}(v)&\quad\hbox{ if }u(t,x)\ge 0\,,
\\
-\indc_{[u(t,x),0]}(v)&\quad\hbox{ if }u(t,x)< 0\,,
\end{array}\right.
\ee
while
\be\lb{Def-fin}
f^{in}(x,v):=\left\{\begin{array}{ll}
+\indc_{[0,u^{in}(x)]}(v)&\quad\hbox{ if }u^{in}(x)\ge 0\,,
\\
-\indc_{[u^{in}(x),0]}(v)&\quad\hbox{ if }u^{in}(x)< 0\,.
\end{array}\right.
\ee

\begin{Thm}\lb{T-RegEffCvx}
Let $a\in C^2(\bR)$ satisfy (\ref{Hyp-a}), let $m$ be a signed Radon measure on $\bR_+\times\bR\times\bR$. Let $u^{in}\in L^\infty(\bR)$ and let 
$u\in L^\infty([0,T]\times\bR)$ satisfy (\ref{KinForm}). Then
$$
u\in B^{1/p,p}_{\infty,loc}(\bR_+^*\times\bR)\quad\hbox{ with }p=2+\b\,.
$$
More precisely, for all $\eps>0$ and all $\xi\in[-\eps,\eps]$, one has
$$
\ba
\frac{\a_U\b^2}{(\b+1)(\b+2)}\int_0^T\int_{\bR}\chi(t,x)^2|u(t,x+\xi)-u(t,x)|^{2+\b}dxdt
\\
\le
2|\xi|(\|\chi\|_{L^\infty}+\|\d_x\chi\|_{L^1})\left(2U\|X\|_{L^1}+\|\chi\|_{L^\infty}\int_0^T\int_{-R-\eps}^{R+\eps}\int_{-U}^Ud|m|\right)
\ea
$$
and, for all $0<\tau<\eps$
$$
\ba
\frac{\a_U\b^2}{(\b+1)(\b+2)}\int_0^T\int_{\bR}\chi(t,x)^2|u(t+\tau,x)-u(t,x)|^{2+\b}dxdt
\\
\le
2\tau(\|\chi\|_{L^\infty}+\|\d_x\chi\|_{L^1})\left(\|a'\|_{L^1(-U,U)}\|X\|_{L^1}+\|\chi\|_{L^\infty}\int_0^{T+\eps}\int_{-R}^{R}\int_{-U}^U|a'|d|m|\right)
\ea
$$
where $U:=\|u\|_{L^\infty}$ while $X(t,x,v):=(\d_t+a'(v)\d_x)\chi(t,x)$ and $T,R>0$ are such that $\Supp(\chi)\subset[0,T]\times[-R,R]$.
\end{Thm}

The following statement is an obvious consequence of this theorem.

\begin{Cor}
Under the same assumptions as in Theorem \ref{T-RegEffCvx}, one has also
$$
u\in B^{1/p,p}_{\infty,loc}(\bR_+\times\bR)\qquad\hbox{ for each }p\ge 2+\b\,.
$$
\end{Cor}

\begin{proof}
Pick $\chi\in C^\infty_c(\bR_+^*\times\bR)$, and extend $\chi u$ by $0$ to $\bR^2$. Then, for each $p>2+\b$ and $\tau,\xi\in\bR$, one has
$$
\ba
\iint&|\chi u(t+\tau,x+\xi)-\chi u(t,x)|^pdxdt
\\
&\le (2\|\chi u\|_{L^\infty})^{p-2-\b}\iint|\chi u(t+\tau,x+\xi)-\chi u(t,x)|^{2+\b}dxdt\le C(|\tau|+|\xi|)
\ea
$$
since $\chi u\in L^\infty(\bR^2)\times B^{1/(2+\b),2+\b}_{\infty,loc}(\bR^2)$.
\end{proof}

Before giving the proof of this theorem, a few remarks are in order. 

First, the exponent $\b$ in (\ref{Hyp-a}) can be viewed as a measure of the nonlinearity in (\ref{ScalCons}). Indeed, if $a$ is linear, $a'$ is a constant
so that (\ref{Hyp-a}) holds (for $|w-v|<1$) with $\b=+\infty$, and there is no regularizing effect. In other words, the regularizing effect predicted in 
Theorem \ref{T-RegEffCvx} is a consequence of the nonlinearity.

In the case where $a$ satisfies  
\be\lb{StrictConv}
a\in C^2(\bR)\quad\hbox{ with }a''(v)\ge \a>0\,,\quad v\in\bR\,
\ee
P. Lax \cite{Lax57} and O. Oleinik \cite{Oleinik} have proved that, for each initial data $u^{in}\in L^1(\bR)$, the Cauchy problem (\ref{ScalCons}) has a 
unique entropy solution, i.e. a weak solution satisfying 
$$
\d_t\eta(u)+\d_xq(u)\le 0
$$
for each convex $\eta\in C^1(\bR)$ and $q$ defined by (\ref{Def-q}), and that this solution $u$ satisfies $u(t,\cdot)\in L^\infty(\bR)$ for each $t>0$ and
together with the one-sided bound
$$
\d_xu(t,x)\le\frac1{\a t}\,,\qquad t>0\,,\quad x\in\bR\,.
$$
A consequence of this one-sided bound is that $u\in BV_{loc}(\bR_+\times\bR)$. 

In the case where $a(v)=\tfrac12v^2$, for each $p\in[1,\infty]$ and each $\si>1/\max(3,p)$, C. DeLellis and M. Westdickenberg \cite{DeLelWest} have 
proved the existence of a weak solutions $u$ of (\ref{ScalCons}) satisfying the entropy relation (\ref{EntrRel}) for each convex $\eta\in C^2(\bR)$ with 
$q$ defined as in (\ref{Def-q}) and an entropy production rate $\mu$ that is a signed Radon measure on $\bR_+\times\bR$, such that 
$$
u\notin B^{\si,p}_{\infty,loc}(\bR_+^*\times\bR)\,.
$$

Of course, the condition (\ref{StrictConv}) implies that (\ref{Hyp-a}) is verified with $\a_M=\a$ for each $M>0$ and $\b=1$. In that case, Theorem 
\ref{T-RegEffCvx} predicts that weak solutions $u$  of (\ref{ScalCons}) satisfying the entropy relation (\ref{EntrRel}) for each convex $\eta\in C^2(\bR)$ 
with $q$ defined as in (\ref{Def-q}) and an entropy production rate $\mu$ that is a signed Radon measure on $\bR_+\times\bR$ belong to the Besov
space $B^{1/3,3}_{\infty,loc}(\bR_+^*\times\bR)$. According to the result of C. DeLellis and M. Westdickenberg \cite{DeLelWest}, this regularity is
optimal.

\smallskip
A key argument in the proof of Theorem \ref{T-RegEffCvx} is the following inequality.

\begin{Lem}\lb{L-LowBd}
Assume that $a\in C^1(\bR)$ satisfies assumption (\ref{Hyp-a}). For all $u\in\bR$, define
$$
\cM_u(v):=\left\{\begin{array}{ll}
+\indc_{[0,u]}(v)&\quad\hbox{ if }u\ge 0\,,
\\
-\indc_{[u,0]}(v)&\quad\hbox{ if }u< 0\,.
\end{array}\right.
$$
and
$$
\Dlt(u,\bar u):=\!\iint\indc_{\bR_+}(v\!-\!w)(a'(v)\!-\!a'(w))(\cM_u(v)\!-\!\cM_{\bar u}(v))(\cM_u(w)\!-\!\cM_{\bar u}(w))dvdw\,.
$$

Then, for each $V>0$ and $\bar u,u\in[-V,V]$, one has
$$
\Dlt(u,\bar u)\ge\frac{\a_V\b^2}{(\b+1)(\b+2)}|u-\bar u|^{2+\b}\,.
$$
\end{Lem}

The proof of this inequality is deferred until after the proof of Theorem \ref{T-RegEffCvx}.

\begin{proof}[Proof of Theorem \ref{T-RegEffCvx}]
Pick $\chi\in C^\infty_c(\bR_+\times\bR)$, with support in $[0,T]\times[-R,R]$, and let $U=\|u\|_{L^\infty([0,T]\times\bR)}$.

We first establish the regularity in the space variable $x$. As in the proof of Theorem \ref{T-VelAv}, set
$$
A(t,x,v):=\chi(t,x)\cD_x^hf(t,x,v)\,,\qquad D(t,x,w):=\chi(t,x)\cD_x^hf(t,x,w)\,.
$$
Since $f\equiv f(t,x,v)$ satisfies (\ref{KinEq}), the function $A$ and $D$ defined above satisfy (\ref{2x2Syst}) with
$$
B(t,x,v):=a'(v)\chi(t,x)\cD_x^hf(t,x,v)\,,\qquad E(t,x,w):=a'(w)\chi(t,x)\cD_x^hf(t,x,w)
$$
and
$$
\left\{
\ba
C(t,x,v)&:=\chi(t,x)\d_v\cD_x^hm(t,x,v)\,+X(t,x,v)\cD_x^hf(t,x,v)\,,
\\
F(t,x,w)&:=\chi(t,x)\d_w\cD_x^hm(t,x,w)+X(t,x,w)\cD_x^hf(t,x,w)\,,
\ea
\right.
$$
with the notation
$$
X(t,x,v):=\d_t\chi(t,x)+a'(v)\d_x\chi(t,x)\,.
$$

Multiplying each side of the interaction identity (\ref{InterIdent}) by $\indc_{\bR_+}(v-w)$ and integrating both sides of the resulting equality in the 
variables $v$ and $w$,
\be\lb{Q=}
\ba
Q\!:=&\!\iint_{\bR\times\bR}\!\indc_{\bR_+}(v\!-\!w)\!\iint_{\bR\times\bR}\!(A(t,x,v)E(t,x,w)\!-\!B(t,x,v)D(t,x,w))dxdtdvdw
\\
=&-\iint_{\bR\times\bR}\indc_{\bR_+}(v-w)\iint_{\bR\times\bR}C(t,x,v)\left(\int_x^\infty D(t,y,w)dy\right)dxdtdvdw
\\
&-\iint_{\bR\times\bR}\indc_{\bR_+}(v-w)\iint_{\bR\times\bR}F(t,y,w)\left(\int_{-\infty}^y A(t,x,v)dx\right)dydtdvdw\,.
\ea
\ee
The right hand side of this identity involves essentially terms of two different kinds:
\be\lb{Q1=}
Q_1=\iint\indc_{\bR_+}(v-w)\iint_{\bR\times\bR}X(t,x,v)\cD_x^hf(t,x,v)\left(\int_x^\infty D(t,y,w)dy\right)dxdtdvdw\,,
\ee
the other being
\be\lb{Q2=}
Q_2=\iint\indc_{\bR_+}(v-w)\iint_{\bR\times\bR}\chi(t,x)\d_v\cD_x^hm(t,x,v)\left(\int_x^\infty D(t,y,w)dy\right)dxdtdvdw\,.
\ee

Proceeding as in the proof of Theorem \ref{T-VelAv}, we put the inner integral in the form
$$
\ba
\int_x^\infty D(t,y,w)dy=\int_y^\infty\cD_y^h(\chi(t,y)f(t,y,w))dy-\int_x^\infty f(t,y+h,w)\cD_y^h\chi(t,y)dy
\\
=-\int_{x}^{x+h}\chi(t,y)f(t,y,w)dy-\int_x^\infty f(t,y+h,w)\left(\int_0^h\d_x\chi(t,y+z)dz\right)dy\,,
\ea
$$
so that
\be\lb{EstimD2}
\left|\int_x^\infty D(t,y,w)dy\right|\le(\|\chi\|_{L^\infty}+\|\d_x\chi\|_{L^1})|h|\indc_{[-U,U](w)}\,.
\ee

Thus
$$
|Q_1|\le K_1|h|\,,
$$
with
$$
\ba
(\|\chi\|_{L^\infty}+\|\d_x\chi\|_{L^1})\iint\|X(\cdot,\cdot,v)\|_{L^1}\indc_{[-U,U]}(v)\indc_{[-U,U]}(w)dvdw
\\
\le 2U(\|\chi\|_{L^\infty}+\|\d_x\chi\|_{L^1})\|X\|_{L^1}=:K_1\,.
\ea
$$

On the other hand, for each $h\in[-\eps,\eps]$
$$
\ba
Q_2=\iint\indc_{\bR_+}(v-w)\iint_{\bR\times\bR}\chi(t,x)\d_v\cD_x^hm(t,x,v)\left(\int_x^\infty D(t,y,w)dy\right)dxdtdvdw
\\
=-\int_{\bR}\iint_{\bR\times\bR}\chi(t,x)\cD_x^hm(t,x,w)\left(\int_x^\infty D(t,y,w)dy\right)dxdtdw\,,
\ea
$$
so that
$$
|Q_2|\le K_2|h|
$$
with
$$
K_2:=2(\|\chi\|_{L^\infty}+\|\d_x\chi\|_{L^1})\|\chi\|_{L^\infty}\int_0^T\int_{-R-\eps}^{R+\eps}\int_{-U}^U|m|\,.
$$

In conclusion, one has, by the same argument as in the proof of Theorem \ref{T-VelAv}
\be\lb{UpQ}
|Q|\le 2(K_1+K_2)|h|\,.
\ee

On the other hand, according to Lemma \ref{L-LowBd}
\be\lb{LowQ}
\ba
Q&=\int_0^T\int_{\bR}\chi(t,x)^2\Dlt(u(t,x),u(t,x+h))dxdt
\\
&\ge\frac{\a_U\b^2}{(\b+1)(\b+2)}\int_0^T\int_{\bR}\chi(t,x)^2|\cD_x^hu(t,x)|^{2+\b}dxdt\,.
\ea
\ee

Putting together (\ref{LowQ}) and (\ref{UpQ}) shows that
$$
u\in L^{2+\b}_{loc}(\bR_+^*;B^{1/(2+\b),2+\b}_{\infty,loc}(\bR))\,.
$$

Now for the time regularity. Following the proof of Theorem \ref{T-VelAv}, pick $0<h<\eps$ and set
$$
A(t,x,v):=\chi(t,x)\cD_t^hf(t,x,v)\,,\qquad D(t,x,w):=\chi(t,x)\cD_t^hf(t,x,w)\,.
$$
Since $f\equiv f(t,x,v)$ satisfies (\ref{KinEq}), the function $A$ and $D$ defined above satisfy (\ref{2x2Syst}) with
$$
B(t,x,v):=a'(v)\chi(t,x)\cD_t^hf(t,x,v)\,,\qquad E(t,x,w):=a'(w)\chi(t,x)\cD_t^hf(t,x,w)\,,
$$
and
$$
\left\{
\ba
C(t,x,v)&:=\chi(t,x)\d_v\cD_x^hm(t,x,v)\,+X(t,x,v)\cD_t^hf(t,x,v)\,,
\\
F(t,x,w)&:=\chi(t,x)\d_w\cD_x^hm(t,x,w)+X(t,x,w)\cD_t^hf(t,x,w)\,,
\ea
\right.
$$

Multiplying each side of the interaction identity (\ref{InterIdent2}) by $\indc_{\bR_+}(v-w)$ and integrating both sides of the resulting equality in the 
variables $v$ and $w$,
\be\lb{S=}
\ba
S:=\iint_{\bR\times\bR}&(AE-BD)(t,x)dxdt
\\
=&\iint\indc_{\bR_+}(v-w)\iint_{\bR\times\bR}C(s,x,v)\left(\int_s^\infty E(t,x,w)dt\right)dxdsdvdw
\\
&+\iint\indc_{\bR_+}(v-w)\iint_{\bR\times\bR}F(t,x,w)\left(\int_{-\infty}^t B(s,x,v)ds\right)dxdtdvdw\,.
\ea
\ee
As before, the right hand side of this identity involves terms of two different kinds:
\be\lb{S1=}
S_1=\iint\indc_{\bR_+}(v-w)\iint_{\bR\times\bR}X(s,x,v)\cD_s^hf(s,x,v)\left(\int_s^\infty E(t,x,w)dt\right)dxdsdvdw\,,
\ee
the other being
\be\lb{S2=}
S_2=\iint\indc_{\bR_+}(v-w)\iint_{\bR\times\bR}\chi(s,x)\d_v\cD_x^hm(s,x,v)\left(\int_s^\infty E(t,x,w)dt\right)dxdsdvdw\,.
\ee

The inner integral
$$
\ba
\int_s^\infty\!\! E(t,x,w)dt\!=\!\int_s^\infty\!\!\cD_s^h(\chi(t,x)a'(w)f(t,x,w))dt-\!\int_x^\infty\!\! f(t+h,x,w)\cD_t^h\chi(t,x)dt
\\
=\int_{s}^{s+h}\chi(t,y)a'(w)f(t,y,w)dy\!-\!\int_s^\infty a'(w)f(t+h,x,w)\left(\int_0^h\d_x\chi(t+\tau,x)d\tau\right)dt\,,
\ea
$$
so that
\be\lb{EstimE2}
\left|\int_s^\infty E(t,x,w)dt\right|\le(\|\chi\|_{L^\infty}+\|\d_t\chi\|_{L^1})h|a'(w)|\indc_{[-U,U](w)}\,.
\ee

Thus
$$
|S_1|\le L_1h\,,
$$
with
$$
\ba
(\|\chi\|_{L^\infty}+\|\d_t\chi\|_{L^1})\iint\|X(\cdot,\cdot,v)\|_{L^1}|a'(w)|\indc_{[-U,U]}(v)\indc_{[-U,U]}(w)dvdw
\\
\le(\|\chi\|_{L^\infty}+\|\d_t\chi\|_{L^1})\|X\|_{L^1}\|a'\|_{L^1(-U,U)}=:L_1\,.
\ea
$$

Moreover
$$
\ba
S_2=\iint\indc_{\bR_+}(v-w)\iint_{\bR\times\bR}\chi(s,x)\d_v\cD_s^hm(s,x,v)\left(\int_s^\infty E(t,x,w)dt\right)dxdsdvdw
\\
=-\int_{\bR}\iint_{\bR\times\bR}\chi(t,x)\cD_s^hm(s,x,w)\left(\int_s^\infty E(t,x,w)dt\right)dxdsdw\,,
\ea
$$
so that
$$
|S_2|\le L_2h
$$
with
$$
L_2:=2(\|\chi\|_{L^\infty}+\|\d_t\chi\|_{L^1})\|\chi\|_{L^\infty}\int_0^{T+\eps}\int_{-R}^{R}\int_{-U}^U|a'||m|\,.
$$

Collecting the bounds on $S_1$ and $S_2$, one has
\be\lb{UpS}
|S|\le 2(L_1+L_2)h\,.
\ee

On the other hand, according to Lemma \ref{L-LowBd}
\be\lb{LowS}
\ba
S&=\int_0^T\int_{\bR}\chi(t,x)^2\Dlt(u(t,x),u(t+h,x))dxdt
\\
&\ge\frac{\a_U\b^2}{(\b+1)(\b+2)}\int_0^T\int_{\bR}\chi(t,x)^2|\cD_t^hu(t,x)|^{2+\b}dxdt\,.
\ea
\ee

Putting together (\ref{LowS}) and (\ref{UpS}) shows that
$$
u\in L^{2+\b}_{loc}(\bR;B^{1/(2+\b),2+\b}_{\infty,loc}(\bR_+^*))\,,
$$
which concludes the proof.
\end{proof}

\begin{proof}[Proof of Lemma \ref{L-LowBd}]
Assume that $\bar u\le u$. Then
$$
\cM_u(v)-\cM_{\bar u}(v)=\left\{\begin{array}{ll}
\indc_{(\bar u,u]}(v)&\quad\hbox{ if }u\ge\bar u\ge 0\,,
\\
\indc_{[\bar u,u]}(v)+\indc_{0}(v)&\quad\hbox{ if }u\ge 0>\bar u\,,
\\
\indc_{[\bar u,u)}(v)&\quad\hbox{ if }0>u\ge\bar u\,.
\end{array}
\right.
$$

Therefore, if $-V\le\bar u<u\le V$, one has
$$
\ba
\Dlt(u,\bar u)&=\iint\indc_{\bR_+}(v-w)(a'(v)-a'(w))\indc_{(\bar u,u)}(v)\indc_{(\bar u,u)}(w)dvdw
\\
&=\int_{\bar u}^u\left(\int_w^u(a'(v)-a'(w))dv\right)dw
\\
&\ge\a_V\int_{\bar u}^u\int_w^u(v-w)^\b dvdw
\\
&\ge\frac{\a_V\b}{\b+1}\int_{\bar u}^u(u-w)^{\b+1}dw
\\
&=\frac{\a_V\b^2}{(\b+1)(\b+2)}(u-\bar u)^{\b+2}\,,
\ea
$$
which is the announced lower bound when $\bar u<u$.

On the other hand, 
$$
(\cM_u(v)-\cM_{\bar u}(v))(\cM_u(w)-\cM_{\bar u}(w))=(\cM_{\bar u}(v)-\cM_u(v))(\cM_{\bar u}(w)-\cM_u(w))
$$
for all $u,\bar u,v,w\in\bR$, so that 
$$
\Dlt(u,\bar u)=\Dlt(\bar u,u)\,.
$$
With the previous inequality, this establishes the announced lower bound for all $u,\bar u\in\bR$.
\end{proof}

\smallskip
Observe that, for all $u,\bar u\in\bR$
$$
\ba
(\cM_u(v)-\cM_{\bar u}(v))&(\cM_u(w)-\cM_{\bar u}(w))
\\
&\ge\indc_{(\inf(\bar u,u),\sup(\bar u,u))}(v)\indc_{(\inf(\bar u,u),\sup(\bar u,u))}(w)\ge 0
\ea
$$
for all $v,w\in\bR$, so that $\cM_u$ is an example of function satisfying the assumption (\ref{Hyp-f}) above.

\section{Regularizing Effect with One Convex Entropy}\lb{S-ConsLawOne}

In the case where the entropy condition is known to hold for one convex entropy, and with an entropy production rate that is a Radon measure with 
possibly undefinite sign, the following result was obtained by the first author:

\begin{Thm}\cite{GolseHyp08, GolseICMP09}\lb{T-RegFG}
Let $a\in C^1(\bR)$ satisfy (\ref{Hyp-a}), and let $\mu$ be a signed Radon measure on $\bR_+\times\bR$. Assume that the Cauchy problem for the
scalar conservation law (\ref{ScalCons}) has a weak solution $u\in L^\infty(\bR_+\times\bR)$ satisfying
\be\lb{EntrRel}
\d_t\eta(u)+\d_xq(u)=-\mu
\ee
for some $\eta\in C^1(\bR)$ such that there exists $\b'\ge 1$ for which, given any $V>0$, there exists $\eta_{0V}>0$ such that 
\be\lb{Hyp-eta}
\eta'(v)-\eta'(w)\ge\eta_{0,V}(v-w)^{\b'}\,,\quad\hbox{ whenever }-V\le w<v\le V\,,
\ee
with $q$ defined by (\ref{Def-q}). Then
$$
u\in B^{1/p,p}_{\infty,loc}(\bR_+^*\times\bR)\quad\hbox{ with }p=\b+\b'+2\,.
$$
More precisely, for all $\eps>0$ and each $\xi\in[-\eps,\eps]$
$$
\ba
\frac{\a_U\eta_{0U}(\b+\b')}{(\b+\b'+1)(\b+\b'+2)}\iint_{\bR\times\bR}\chi(t,z)^2|u(t,x+\xi)-u(t,x)|^{\b+\b'+2}dxdt
\\
\le\left(M_1+M_2+2U\|\chi\|_{L^\infty}(\|\chi\|_{L^\infty}+\|\d_x\chi\|_{L^1})\int_0^T\int_{-R-\eps}^{R+\eps}d|\mu|\right)
\ea
$$
and, for each $\tau\in[0,\eps]$,
$$
\ba
\frac{\a_U\eta_{0U}(\b+\b')}{(\b+\b'+1)(\b+\b'+2)}\iint_{\bR\times\bR}\chi(t,z)^2|u(t+\tau,x)-u(t,x)|^{\b+\b'+2}dxdt
\\
\le\left(N_1+N_2+2U\|\chi\|_{L^\infty}(\|\chi\|_{L^\infty}+\|\d_t\chi\|_{L^1})\int_0^{T+\eps}\int_{-R}^{R}d|\mu|\right)
\ea
$$
where $U=\|\chi\|_{L^\infty}$, while $M_1$, $M_2$, $N_1$, $N_2$ are defined in (\ref{DefM1}), (\ref{DefM2}), (\ref{DefN1}) and (\ref{DefN2}) respectively
and $T,R>0$ are chosen so that $\Supp(\chi)\subset[0,T]\times[-R,R]$.
\end{Thm}

The proof of this result in \cite{GolseHyp08,GolseICMP09} is based on an argument that is reminiscent of Tartar's compensated compactness method 
for proving the convergence of the vanishing viscosity method for scalar conservation laws \cite{TartarHW79}.

Whenever $a,\eta\in C^2(\bR)$ satisfy
$$
a''(v)\ge\a>0\hbox{ and }\eta''(v)\ge\eta_0>0\quad\hbox{ for all }v\in\bR\,,
$$
one can take $\b=\b'=1$ in (\ref{Hyp-a}) and (\ref{Hyp-eta}), and Theorem \ref{T-RegFG} predicts that $u\in B^{1/4,4}_{\infty,loc}(\bR_+^*\times\bR)$. 

In any case, the regularity obtained in Theorem \ref{T-RegFG} in the case where the weak solution $u$ satisfies only one entropy relation with entropy 
production rate that is a signed Radon measure also belongs to the DeLellis-Westdickenberg \cite{DeLelWest} optimal regularity class --- i.e. the spaces
$B^{1/p,p}_{\infty,loc}$ for $p\ge 3$.

A remarkable result due to Panov \cite{Panov} states that, if $u\in L^\infty(\bR_+\times\bR)$ is a weak solution of (\ref{ScalCons}) with $a\in C^2(\bR)$ 
such that $a''>0$ satisfying one entropy condition (\ref{EntrRel}) with $\eta\in C^2(\bR)$ such that $\eta''>0$ and $\mu\ge 0$, then it is the unique
entropy solution of (\ref{ScalCons}). In particular, it satisfies (\ref{EntrRel}) for all $\eta\in C^2(\bR)$ such that $\eta''>0$, with nonnegative entropy 
production $\mu$. Panov's result was subsequently somewhat generalized by DeLellis, Otto and Westdickenberg \cite{DeLelOttoWest}. However, we
do not know whether any weak solution of (\ref{ScalCons}) with $a\in C^2(\bR)$ such that $a''>0$ satisfying one entropy condition (\ref{EntrRel}) with 
$\eta\in C^2(\bR)$ such that $\eta''>0$ and $\mu$ a signed Radon measure must satisfy (\ref{EntrRel}) for all convex entropies, i.e. whether such a
solution satisfies the assumptions of Theorem \ref{T-RegEffCvx}. Therefore, Theorem \ref{T-RegFG} seems to be of independent interest.

\smallskip
Below, we give a new proof of Theorem \ref{T-RegFG} based on the interaction identity of section \ref{S-Interact} instead of the variant of compensated
compactness used in \cite{GolseHyp08,GolseICMP09}.

\begin{Lem}\lb{L-Tartar<}
Assuming that the functions $a$, $\eta$ and $q$ belong to $C^1(\bR)$ and satisfy (\ref{Hyp-a}) and (\ref{Hyp-eta}) while $q'=a'\eta'$, one has, for each $V>0$
$$
\ba
(w-v)(q(w)-q(v))-(a(w)-a(v))(\eta(w)-\eta(v))&
\\
\ge\frac{\a_V\eta_{0V}(\b+\b')}{(\b+\b'+1)(\b+\b'+2)}|w-v|^{\b+\b'+2}&\,,
\ea
$$
whenever $v,w\in[-V,V]$.
\end{Lem}

In the most general case where $a$ and $\eta$ are $C^1$ convex functions, the inequality
$$
(w-v)(q(w)-q(v))-(a(w)-a(v))(\eta(w)-\eta(v))\ge0\,,\quad v,w\in\bR\,.
$$
is stated without proof by L. Tartar: see Remark 30 in \cite{TartarHW79}. 

For a proof in the case where $a,\eta\in\bC^2(\bR)$ with 
$$
\eta''(v)\ge\eta_0>0\quad\hbox{ and }a''(v)\ge\a>0\quad\hbox{ for all }v\in\bR\,,
$$
corresponding with assumptions (\ref{Hyp-a}) and (\ref{Hyp-eta}) with $\b=\b'=1$, see  Lemma 2.3 in \cite{GolseHyp08}.

\begin{proof}
One has
$$
\ba
(w-v)(q(w)-q(v))&-(a(w)-a(v))(\eta(w)-\eta(v))
\\
&=\int_v^w\int_v^w\eta'(\zeta)(a'(\zeta)-a'(\xi))d\xi d\zeta
\\
&=\tfrac12\int_v^w\int_v^w(\eta'(\zeta)-\eta(\xi)(a'(\zeta)-a'(\xi))d\xi d\zeta 
\\
&\ge\tfrac12\eta_{0V}\a_V\int_v^w\int_v^w(\zeta-\xi)^{\b+\b'}d\xi d\zeta
\\
&=\tfrac{\a_V\eta_{0V}(\b+\b')}{(\b+\b'+1)(\b+\b'+2)}|w-v|^{\b+\b'+2}\,.
\ea
$$ 
\end{proof}

\begin{proof}[Proof of Theorem \ref{T-RegFG}]
Pick $\chi\in C^\infty_c(\bR_+\times\bR)$, with support in $]0,T]\times[-R,R]$, and let $U:=\|u\|_{L^\infty(\bR_+\times\bR)}$.

As above, we first establish the regularity in the space variable $x$. Set
$$
A(t,x):=\chi(t,x)\cD_x^hu(t,x)\,,\qquad D(t,x):=\chi(t,x)\cD_x^h\eta(u)(t,x)\,.
$$
Since $u\equiv f(t,x)$ satisfies (\ref{ScalCons})-(\ref{EntrRel}), the function $A$ and $D$ defined above satisfy (\ref{2x2Syst}) with
$$
B(t,x):=\chi(t,x)\cD_x^ha(u)(t,x)\,,\qquad E(t,x):=\chi(t,x)\cD_x^hq(u)(t,x)\,,
$$
and
$$
\left\{
\ba
C(t,x)&:=\cD_x^hu(t,x)\d_t\chi(t,x)+\cD_x^ha(u)(t,x)\d_x\chi(t,x)\,,
\\
F(t,x)&:=\cD_x^h\eta(u)(t,x)\d_t\chi(t,x)+\cD_x^hq(u)(t,x)\d_x\chi(t,x)-\chi(t,x)\cD_x^h\mu\,.
\ea
\right.
$$

Using the identity (\ref{InterIdent}) shows that
$$
\ba
\iint_{\bR\times\bR}\chi(t,z)^2&(\cD_x^hu\cD_x^hq(u)-\cD_x^ha(u)\cD_x^h\eta(u))(t,z)dzdt
\\
=&-\iint_{\bR\times\bR}C(t,x)\left(\int_x^\infty D(t,y)dy\right)dxdt
\\
&-\iint_{\bR\times\bR}F(t,y)\left(\int_{-\infty}^y A(t,x)dx\right)dydt\,.
\ea
$$

Proceeding as in the proof of Theorem \ref{T-RegEffCvx}, we see that
$$
\ba
\int_{-\infty}^yA(t,x)dx=\int_{-\infty}^y\cD_x^h(\chi u)(t,x)dx-\int_{-\infty}^yu(t,x+h)\cD_x^h\chi(t,x)dx
\\
\int_y^{y+h}\chi u(t,x)dx-\int_{-\infty}^yu(t,x+h)\left(\int_0^h\d_x\chi(t,x+z)dz\right)dx\,,
\ea
$$
so that
$$
\left|\int_{-\infty}^yA(t,x)dx\right|\le\|u\|_{L^\infty}(\|\chi\|_{L^\infty}+\|\d_x\chi\|_{L^1})|h|\,.
$$
Likewise
$$
\left|\int_x^\infty D(t,y)dy\right|\le\|\eta(u)\|_{L^\infty}(\|\chi\|_{L^\infty}+\|\d_x\chi\|_{L^1})|h|\,.
$$

Therefore, assuming that $|h|\le\eps$, one has
$$
\left|\iint_{\bR\times\bR}C(t,x)\left(\int_x^\infty D(t,y)dy\right)dxdt\right|\le M_1|h|\,,
$$
with
\be\lb{DefM1}
M_1:=
(2\|u\|_{L^\infty}\|\d_t\chi\|_{L^1}+\|a(u)\|_{L^\infty}\|\d_x\chi\|_{L^1})\|\eta(u)\|_{L^\infty}(\|\chi\|_{L^\infty}+\|\d_x\chi\|_{L^1})\,,
\ee
while
$$
\left|\iint_{\bR\times\bR}F(t,y)\left(\int_{-\infty}^y A(t,x)dx\right)dydt\right|\le(M_2+M_3)|h|\,,
$$
with
\be\lb{DefM2}
M_2:=(2\|\eta(u)\|_{L^\infty}\|\d_t\chi\|_{L^1}+\|q(u)\|_{L^\infty}\|\d_x\chi\|_{L^1})\|u\|_{L^\infty}(\|\chi\|_{L^\infty}+\|\d_x\chi\|_{L^1})
\ee
and
$$
M_3:=2\|u\|_{L^\infty}(\|\chi\|_{L^\infty}+\|\d_x\chi\|_{L^1})\|\chi\|_{L^\infty}\int_0^T\int_{-R-\eps}^{R+\eps}|\mu|\,.
$$

Thus
$$
\left|\iint_{\bR\times\bR}\chi(t,z)^2(\cD_x^hu\cD_x^hq(u)-\cD_x^ha(u)\cD_x^h\eta(u))(t,z)dzdt\right|\le (M_1+M_2+M_3)|h|\,.
$$

On the other hand, by Lemma \ref{L-Tartar<}
\be\lb{LowBd}
(\cD_x^hu\cD_x^hq(u)-\cD_x^ha(u)\cD_x^h\eta(u))\ge\frac{\a_U\eta_{0U}(\b+\b')}{(\b+\b'+1)(\b+\b'+2)}|\cD_x^hu|^{\b+\b'+2}\,,
\ee
so that the inequality above entails the estimate
$$
\ba
\iint_{\bR\times\bR}\chi(t,z)^2|\cD_x^hu|^{\b+\b'+2}(t,z)dzdt&
\\
\le\frac{(\b+\b'+1)(\b+\b'+2)}{\a_U\eta_{0U}(\b+\b')}(M_1+M_2+M_3)|h|&\,,
\ea
$$
showing that 
$$
u\in L^{\b+\b'+2}_{loc}(\bR_+^*;B^{1/(\b+\b'+2),\b+\b'+2}_{\infty,loc}(\bR))\,.
$$

As for the time regularity, pick $h\in[0,\eps]$ and set
$$
A(t,x):=\chi(t,x)\cD_t^hu(t,x)\,,\qquad D(t,x):=\chi(t,x)\cD_t^h\eta(u)(t,x)\,.
$$
Since $u\equiv f(t,x)$ satisfies (\ref{ScalCons})-(\ref{EntrRel}), the function $A$ and $D$ defined above satisfy (\ref{2x2Syst}) with
$$
B(t,x):=\chi(t,x)\cD_t^ha(u)(t,x)\,,\qquad E(t,x):=\chi(t,x)\cD_t^hq(u)(t,x)\,,
$$
and
$$
\left\{
\ba
C(t,x)&:=\cD_t^hu(t,x)\d_t\chi(t,x)+\cD_t^ha(u)(t,x)\d_x\chi(t,x)\,,
\\
F(t,x)&:=\cD_t^h\eta(u)(t,x)\d_t\chi(t,x)+\cD_t^hq(u)(t,x)\d_x\chi(t,x)-\chi(t,x)\cD_t^h\mu\,.
\ea
\right.
$$

At this point, we use the identity (\ref{InterIdent2}) which shows that
$$
\ba
\iint_{\bR\times\bR}\chi(t,z)^2&(\cD_t^hu\cD_t^hq(u)-\cD_t^ha(u)\cD_t^h\eta(u))(t,z)dzdt
\\
=&\iint_{\bR\times\bR}C(s,x)\left(\int_s^\infty D(t,x)dt\right)dxds
\\
&+\iint_{\bR\times\bR}F(t,x)\left(\int_{-\infty}^t A(s,x)ds\right)dxdt\,.
\ea
$$

As above
$$
\ba
\int_{-\infty}^tA(s,x)ds=\int_{-\infty}^t\cD_x^h(\chi u)(s,x)ds-\int_{-\infty}^tu(s+h,x)\cD_t^h\chi(s,x)ds&
\\
\int_t^{t+h}\chi u(s,x)ds-\int_{-\infty}^tu(s+h,x)\left(\int_0^h\d_x\chi(s+\tau,x)d\tau\right)ds&\,,
\ea
$$
so that
$$
\left|\int_{-\infty}^tA(s,x)ds\right|\le\|u\|_{L^\infty}(\|\chi\|_{L^\infty}+\|\d_t\chi\|_{L^1})h\,.
$$
Likewise
$$
\left|\int_s^\infty D(t,x)dt\right|\le\|\eta(u)\|_{L^\infty}(\|\chi\|_{L^\infty}+\|\d_t\chi\|_{L^\infty})h\,.
$$

Therefore
$$
\left|\iint_{\bR\times\bR}C(s,x)\left(\int_s^\infty D(t,x)ds\right)dxdt\right|\le N_1|h|
$$
with
\be\lb{DefN1}
N_1:=
(2\|u\|_{L^\infty}\|\d_t\chi\|_{L^1}+\|a(u)\|_{L^\infty}\|\d_x\chi\|_{L^1})\|\eta(u)\|_{L^\infty}(\|\chi\|_{L^\infty}+\|\d_t\chi\|_{L^1})\,,
\ee
while
$$
\left|\iint_{\bR\times\bR}F(t,x)\left(\int_{-\infty}^tA(s,x)ds\right)dxdt\right|\le(M_2+M_3)|h|\,,
$$
with
\be\lb{DefN2}
N_2:=(2\|\eta(u)\|_{L^\infty}\|\d_t\chi\|_{L^1}+\|q(u)\|_{L^\infty}\|\d_x\chi\|_{L^1})\|u\|_{L^\infty}(\|\chi\|_{L^\infty}+\|\d_t\chi\|_{L^1})
\ee
and
$$
N_3:=2\|u\|_{L^\infty}(\|\chi\|_{L^\infty}+\|\d_t\chi\|_{L^1})\|\chi\|_{L^\infty}\int_0^{T+\eps}\int_{-R}^R|\mu|\,.
$$

Using the inequality (\ref{LowBd}), we conclude that
$$
\ba
\iint_{\bR\times\bR}\chi(t,z)^2|\cD_t^hu|^{\b+\b'+2}(t,z)dzdt&
\\
\le\frac{(\b+\b'+1)(\b+\b'+2)}{\a_U\eta_{0U}(\b+\b')}(N_1+N_2+N_3)h&\,,
\ea
$$
so that
$$
u\in L^{\b+\b'+2}_{loc}(\bR;B^{1/(\b+\b'+2),\b+\b'+2}_{\infty,loc}(\bR_+^*))\,,
$$
which concludes the proof.
\end{proof}



\begin{thebibliography}{99}

\bibitem{Agosh}
V. Agoshkov:
\textit{Spaces of functions with differential-difference characteristics and the smoothness of solutions of the transport equation},
Dokl. Akad. Nauk SSSR \textbf{276} (1984), no. 6, 1289--1293.

\bibitem{ArseSR}
D. Arsenio, L. Saint-Raymond:
\textit{Concentrations in kinetic transport equations and hypoellipticity},
preprint arXiv:1005.1547

\bibitem{BeniCran}
P. B\'enilan, M. Crandall:
\textit{Regularizing effects of homogeneous evolution equations}. In ``Contributions to analysis and geometry",  23--39, 
Johns Hopkins Univ. Press, Baltimore, Md., 1981. 

\bibitem{BerthJun}
F. Berthelin, S. Junca:
\textit{Averaging lemmas with a force term in the transport equation},
J. Math. Pures Appl. (9) \textbf{93} (2010), 113--131.

\bibitem{Bony}
J.-M. Bony:
\textit{Solutions globales born\'ees pour les mod\`eles discrets de l'\'equation de Boltzmann, en dimension 1 d'espace},
Journ\'ees ``Equations aux deriv\'ees partielles'' (Saint Jean de Monts, 1987), Exp. No. XVI, 10 pp., Ecole Polytech., Palaiseau, 1987.

\bibitem{BouDesv}
F. Bouchut, L. Desvillettes:
\textit{Averaging lemmas without time Fourier transform and application to discretized kinetic equations},
Proc. Roy. Soc. Edinburgh Sect. A \textbf{129} (1999), no. 1, 19Ð36.

\bibitem{BouGolPul}
F. Bouchut, F. Golse, M. Pulvirenti:
``Kinetic equations and asymptotic theory''. 
Edited and with a foreword by Beno"t Perthame and Laurent Desvillettes. Series in Applied Mathematics (Paris), 4. Gauthier-Villars, Editions Scientifiques 
et M\'edicales Elsevier, Paris, 2000. 

\bibitem{CerciWeakSol}
C. Cercignani:
\textit{Weak solutions of the Boltzmann equation without angle cutoff},
J. Stat. Phys. \textbf{123} (2006), no. 4, 753Ð762.

\bibitem{Dafermos}
C. Dafermos:
``Hyperbolic conservation laws in continuum physics'',
Grundlehren der Mathematischen Wissenschaften, 325. Springer-Verlag, Berlin, 2000. 

\bibitem{DeLelOttoWest}
C. De Lellis, M. Westdickenberg:
\textit{Minimal entropy conditions for Burgers equation},
Quart. Appl. Math. \textbf{62} (2004), 687--700. 

\bibitem{DeLelWest}
C. De Lellis, M. Westdickenberg:
\textit{On the optimality of velocity averaging lemmas},
Ann. Inst. H. Poincar\'e Anal. Non Lin\'eaire \textbf{20} (2003), no. 6, 1075--1085.

\bibitem{DVorePe}
R. DeVore, G. Petrova:
\textit{The averaging lemma},
J. Amer. Math. Soc. \textbf{14} (2001), no. 2, 279Ð296.

\bibitem{dPLM}
R. DiPerna, P.-L. Lions, Y. Meyer:
\textit{$L\sp p$ regularity of velocity averages},
Ann. Inst. H. Poincar\'e Anal. Non Lin\'eaire \textbf{8} (1991), no. 3-4, 271--287.

\bibitem{Glimm65}
J. Glimm:
\textit{Solutions in the large for nonlinear hyperbolic systems of equations}, 
Comm. Pure Appl. Math. \textbf{18} (1965), 697Ð715.

\bibitem{GolseHyp08}
F. Golse:
\textit{Nonlinear regularizing effect for conservation laws}. 
Hyperbolic problems: theory, numerics and applications, 73Ð92, Proc. Sympos. Appl. Math., 67, Part 1, Amer. Math. Soc., Providence, RI, 2009.

\bibitem{GolseICMP09}
F. Golse:
\textit{Nonlinear regularizing effect for hyperbolic partial differential equations}.
XVIth International Congress on Mathematical Physics, 433Ð437, World Sci. Publ., Hackensack, NJ, 2010,

\bibitem{GLPS}
F. Golse, P.-L. Lions, B. Perthame, R. Sentis:
\textit{Regularity of the moments of the solution of a transport equation},
J. Funct. Anal. \textbf{76} (1988), no. 1, 110--125.

\bibitem{GSR}
F. Golse, L. Saint-Raymond:
\textit{Velocity averaging in $L^1$ for the transport equation},
C. R. Math. Acad. Sci. Paris \textbf{334} (2002), no. 7, 557Ð562,

\bibitem{GPS}
F. Golse, B. Perthame, R. Sentis:
\textit{Un r\'esultat de compacit\'e pour les \'equations de transport et application au calcul de la limite de la valeur propre principale d'un op\'erateur de 
transport},
C. R. Acad. Sci. Paris S\'er. I Math. \textbf{301} (1985), no. 7, 341--344.

\bibitem{Ha}
S.-Y. Ha:
\textit{$L^1$ stability estimate for a one-dimensional Boltzmann equation with inelastic collisions}, 
J. Differential Equations \textbf{190} (2003), no. 2, 621Ð642.

\bibitem{HwangTzav}
S. Hwang, A. Tzavaras:
\textit{Kinetic Decomposition of Approximate Solutions to Conservation Laws: Application to Relaxation 
and Diffusion-Dispersion Approximations},
Commun. in Partial Diff. Eq. \textbf{27} (2002), 1229--1254.

\bibitem{JabPerth}
P.-E. Jabin, B. Perthame:
\textit{Regularity in kinetic formulations via averaging lemmas}. 
A tribute to J. L. Lions. ESAIM Control Optim. Calc. Var. \textbf{8} (2002), 761Ð774.

\bibitem{JabVega}
P.-E. Jabin, L. Vega:
\textit{A real space method for averaging lemmas},
J. Math. Pures Appl. (9) \textbf{83} (2004), no. 11, 1309Ð1351.

\bibitem{JinXin}
S. Jin, Z. Xin:
\textit{The Relaxing Schemes for Systems of Conservation Laws in Arbitrary Space Dimensions},
Comm. Pure Appl. Math. \textbf{48} (1995), 235--277.

\bibitem{Lax57}
P. Lax:
\textit{Hyperbolic systems of conservation laws II},
Comm. Pure and Appl. Math. \textbf{10} (1957), 537--566.

\bibitem{Lax73}
P. Lax:
``Hyperbolic systems of conservation laws and the mathematical theory of shock waves'',
Conference Board of the Mathematical Sciences, Regional Conference Series in Applied Mathematics, No. 11. 
SIAM, Philadelphia, Pa., 1973.

\bibitem{LPT-a}
P.-L. Lions, B. Perthame, E. Tadmor:
\textit{A kinetic formulation of multidimensional scalar conservation laws and related equations},
J. Amer. Math. Soc. \textbf{7} (1994), no. 1, 169--191.

\bibitem{Murat}
F. Murat:
\textit{Compacit\'e par compensation},
Ann. Scuola Norm. Sup. Pisa Cl. Sci. (4) \textbf{5} (1978), 489Ð507.

\bibitem{Oleinik}
O. Oleinik:
\textit{On discontinuous solutions of nonlinear differential equations},
Doklady Akad. Nauk SSSR, \textbf{109} (1956), 1098--1101.

\bibitem{OttoTransp}
F. Otto:
\textit{A regularizing effect of nonlinear transport equations},
Quart. Appl. Math. \textbf{56} (1998), 355--375.

\bibitem{Panov}
E. Panov:
\textit{Uniqueness of the solution of the Cauchy problem for a first-order quasilinear equation with an admissible strictly convex entropy},
(Russian) Mat. Zametki \textbf{55} (1994), 116--129; English translation in Math. Notes \textbf{55} (1994), 517Ð525.

\bibitem{PerthBook}
B. Perthame:
``Kinetic formulation of conservation laws'',
Oxford Lecture Series in Mathematics and its Applications, 21. Oxford University Press, Oxford, 2002.

\bibitem{PerthSoug}
B. Perthame, P. Souganidis:
\textit{A limiting case for velocity averaging},
Ann. Sci. ƒcole Norm. Sup. (4) \textbf{31} (1998), no. 4, 591Ð598.

\bibitem{Riemann}
B. Riemann:
\textit{\"Uber die Fortpflanzung ebener Luftwellen von endlicher Schwingungsweite},
Abh. d. K\"onigl. Ges. d. Wiss. zu G\"ottingen, \textbf{8} (1858/59) (Math. Cl), 43--55.

\bibitem{Riviere}
T. Rivi\`ere:
\textit{Parois et vortex en micromagn\'etisme},
JournŽes ``Equations aux deriv\'ees partielles'' (Forges-les-Eaux, 2002), Exp. No. XIV, 15 pp., Ecole Polytech., Palaiseau, 2002.

\bibitem{Schonbek}
M.E.  Schonbek:
\textit{Convergece of Solutions to Nonlinear Dispersive Equations},
Commun. on Partial Diff. Eq. \textbf{7} (1982), 959--1000.

\bibitem{Smoller}
J. Smoller:
``Shock  Waves and Reaction Diffusion Equations'',
Grundlehren der Mathematischen Wissenschaften, 258. Springer-Verlag, New York-Berlin, 1983.

\bibitem{TadTao}
E. Tadmor, T. Tao:
\textit{Velocity averaging, kinetic formulations, and regularizing effects in quasi-linear PDEs},
Comm. Pure Appl. Math. \textbf{60} (2007), no. 10, 1488--1521. 

\bibitem{TartarHW79}
L. Tartar:
\textit{Compensated compactness and applications to partial differential equations},
Nonlinear analysis and mechanics: Heriot-Watt Symposium, Vol. IV, pp. 136--212, Res. Notes in Math., 39, Pitman, Boston, Mass.-London, 1979.

\bibitem{TartarSob}
L. Tartar:
``An introduction to Sobolev spaces and interpolation spaces'',
Lecture Notes of the Unione Matematica Italiana, 3. Springer-Verlag, Berlin; UMI, Bologna, 2007.

\bibitem{TartarKin}
L. Tartar:
``From hyperbolic systems to kinetic theory. A personalized quest'',
Lecture Notes of the Unione Matematica Italiana, 6. Springer-Verlag, Berlin; UMI, Bologna, 2008.

\end{thebibliography}
\end{document}